\documentclass[cmp,draft,numbook]{svjour}
\usepackage{amsmath,amssymb,verbatim}

\spnewtheorem*{cor}{Corollary}{\bf}{\it}
\spnewtheorem*{lem}{Lemma}{\bf}{\it}
\spnewtheorem*{pro}{Proposition}{\bf}{\it}
\spnewtheorem*{teo}{Theorem}{\bf}{\it}

\newcommand\al{\alpha}
\newcommand\cc{d}
\newcommand\dd{\partial}
\newcommand\De{\Delta}
\newcommand\de{\delta}
\newcommand\FF{\mathbb F}
\renewcommand\ge{\geqslant}
\newcommand\ib{\hspace{1pt}\bar\imath\hspace{1pt}}
\newcommand\La{\Lambda}
\newcommand\la{\lambda}
\newcommand\lap{\la^{\ts\prime}}
\newcommand\lcd{,\ldots,}
\renewcommand\le{\leqslant}
\newcommand\ns{\hspace{-.75pt}}
\newcommand\om{\omega}
\newcommand\QQ{\mathbb Q}
\renewcommand\phi{\varphi}
\newcommand\rb{\hspace{1pt}\bar r\hspace{1pt}}
\renewcommand\sb{\hspace{1pt}\bar s\hspace{1pt}} % NEW COMMAND \sb
\newcommand\Sg{\mathfrak{S}}
\newcommand\si{\sigma}
\newcommand\sib{\hspace{1pt}\bar\sigma\hspace{1pt}}
\newcommand\ts{\hspace{0.75pt}}
\newcommand\xh{\widehat{x}}

\newcommand\Ref[1]{(\ref{#1})}

\newcommand\beq{\begin{equation}}
\newcommand\eeq{\end{equation}}
\newcommand\be{\begin{equation*}}
\newcommand\ee{\end{equation*}}
\newcommand\abs[1]{\left|#1\right|}

\newcommand\Cbbd{\mathbb{C}}
\newcommand\Rbbd{\mathbb{R}}
\newcommand\Zbbd{\mathbb{Z}}
\newcommand\Hcal{\mathcal{H}}
\newcommand{\rd}{\mathrm{d}}            % Roman d for differential
\newcommand{\ri}{\mathrm{i}}            % Roman i for imaginary number

\begin{document}
\title{Sekiguchi\ts-Debiard operators at infinity}
\titlerunning{Sekiguchi\ts-Debiard operators}

\author{M.\,L.\,Nazarov and E.\,K.\,Sklyanin}
\authorrunning{Nazarov and Sklyanin}

\institute{Department of Mathematics, University of York, 
York YO10 5DD, United Kingdom }

\date{}%{28 March 2013}

\maketitle

%------------------------------------------------------------------------------

\thispagestyle{empty} % NO FIRST PAGE NUMBER

\begin{abstract}
We construct a family of pairwise commuting operators  
such that the Jack symmetric functions of infinitely many variables 
$x_1,x_2,\,\ldots$ are their eigenfunctions.
These operators are defined as
limits at $N\to\infty$ of
renormalised Sekiguchi\ts-Debiard operators
acting on symmetric polynomials in the variables 
$x_1\lcd x_N$.
They are differential operators
in terms of the power sum variables $p_n=x_1^n+x_2^n+\ldots$ 
and we compute their symbols
by using the Jack reproducing kernel.
Our result yields a hierarchy of commuting Hamiltonians
for the quantum Calogero-Sutherland model with infinite number of
bosonic particles in terms of the collective variables of the model.
Our result also yields the elementary step operators
for the Jack symmetric functions. 

\end{abstract}

\newpage%%%%%%%%%%%%%%%%%%%%%%%%%%%%%%%%%%%%%%%%%%%%%%%%%%%%%%%%%%%%%%%%%%%%%%%

%==============================================================================

\section{Introduction}

%------------------------------------------------------------------------------

\subsection{Calogero-Sutherland model\/}

This quantum model 
describes a system of $N$ bosonic particles on a circle $\Rbbd/\pi\Zbbd$
with the Hamiltonian \cite{C,Su1,Su2}
$$
H_\text{CS}=-\,\frac12\,\sum_i\frac{\dd^2}{\dd q_i^2}
\,+\,
\sum_{i<j}\,\frac{\beta\ts(\beta-1)}{\sin^2(q_i-q_j)}
$$
where $0\le q_1,\ldots,q_N<\pi$.
Being translationally invariant $H_\text{CS}$ commutes~with the momentum
operator
$$
P_\text{\ts CS}=-\,\ri\,\,\sum_j\,\frac{\dd}{\dd q_j}\,.
$$
After eliminating the vacuum factor
$$
\om=\Big|\,\prod_{i<j}\sin(q_i-q_j)\,\Big|^{\,\beta}
$$
and then passing to the exponential variables 
$x_j=\exp\ts(\ts2\ts\ri\ts q_j)$
and the parameter $\al=\beta^{\ts-1}$ more common
in the mathematical literature, the Hamiltonian becomes
$$
\omega^{-1}\circ H_\text{CS}\circ\omega
\,=\,\frac2\al\,H^{(2)}_N+\frac{N^3-N}{6\,\al^2}
$$
where
\beq
\label{hamiltonian}
H^{(2)}_N=\,
\al\,\sum_{i}\, 
\Bigl(x_i\,\frac{\dd}{\dd{x_i}}\,\Bigr)^2+\,
\sum_{i<j}\,\frac{x_i+x_j}{x_i-x_j}\,
\Bigl(x_i\,\frac{\dd}{\dd{x_i}}-x_j\,\frac{\dd}{\dd{x_j}}\Bigr)\,.
\eeq
Respectively, $P_\text{CS}$ gives rise to the operator
$$
H^{\ts(1)}_N=\,
\frac12\,
P_\text{\ts CS}\,=\sum_j\,x_j\,\frac{\dd}{\dd x_j}\,.
$$

The operators $H^{(1)}_N$ and $H^{(2)}_N$ commute and %both 
act on symmetric polynomials of the variables $x_1\lcd x_N\ts$.
It is known that both operators can be included into a quantum
integrable hierarchy, that is into a polynomial ring 
of commuting differential operators with $N$ generators
of orders $1,\ldots,N$ respectively,
called the Sekiguchi\ts-Debiard operators \cite{D,S}.
The Jack symmetric polynomials \cite{M} are joint eigenfunctions of the
hierarchy. They are labelled by partitions of $0,1,2,\ldots$
with no more than $N$ non-zero parts.

In combinatorics it is quite common to extend various symmetric polynomials
to an infinite countable set of variables. These extensions are called 
\emph{symmetric functions\/}.
In particular, the extensions of the Jack symmetric polynomials are
well studied \cite{M}. 
They are labelled by partitions of $0,1,2,\ldots\,\ts$.
However, no explicit expressions 
for higher commuting Hamiltonians corresponding to the infinite set 
of variables have been yet available in the Jack case,
with an exception of a few lower order operators.
The main purpose of the present article is to fill up~the gap, by studying 
the limits of the Sekiguchi\ts-Debiard operators as 
$N\to\infty$ and~giving explicit expressions for the resulting commuting
Hamiltonians.

As another application of our result, we construct 
\emph{elementary step operators\/} for the  
Jack symmetric functions.
In terms of the labels, our operators correspond
to decreasing by $1$
any given non-zero part of a partion,
and to the operation on partitions inverse to that,
see our formulas \eqref{cshift} and \eqref{bshift} respectively.
For~the origins of this construction
see the work \cite{Sk} and references therein. 
For related but different results on the
Jack polynomials see the work \cite{LV}.

%------------------------------------------------------------------------------

\subsection{Collective variables\/}

The standard way to treat the $N=\infty$ case is to rewrite the Hamiltonian
\Ref{hamiltonian} in terms of the power sums
(the ``collective variables'' in the condensed matter physics terminology)
\beq
\label{def-pn}
x_1^n+\ldots+x_N^n
\quad\ \text{where}\ \quad
n=1,\ldots,N
\eeq
and to take the limit at
$N\rightarrow\infty$ afterwards.
Denoting the limit of the power sum \eqref{def-pn} by $p_n$
where $n=1,2,\ldots$ 
the resulting Hamiltonian reads \cite{MP}
\begin{gather}
\notag
H^{\ts(2)}=\lim_{N\to\infty}
\al\,\bigl(\ts H^{\ts(2)}_N-N H^{\ts(1)}_N\ts\bigr)\,=
\\[4pt]
\notag
\sum_{m,n=1}^\infty\biggl(
\al\,(m+n)\,p_m\,p_n\ts\frac{\dd}{\dd\ts p_{m+n}}
+\al^2\ts m\,n\,p_{m+n}\ts\frac{\dd^2}{\dd\ts p_m\ts\dd\ts p_n}\,\biggr)\ +
\\
\label{H2}
(\al-1)\,\sum_{n=1}^\infty\,\al\,n^2\ts p_n\,\frac{\dd}{\dd\ts p_n}\ .
\end{gather}
The first and the second summands 
in the middle line of the above display
are known as
\textit{splitting terms\/} 
and 
\textit{joining terms\/} respectively,
see also \cite{AMOS} and \cite{CJ}. 

Consider the vector space $\La=\Cbbd\ts[\,p_1,p_2,\ldots\ts\,]$
and equip it 
with the operators $a_{\ts n}$ where $n\in\Zbbd\setminus\{0\}\,$,
defined on the polynomials $f\in\La$ by
$$
a_{\ts n}\ts(f)=
\left\{\begin{array}{ccl}
p_{-n}\,f&\quad\text{if}\quad n<0\ts,
\\[6pt]
\al\,n\,{\dd f}/{\dd\ts p_n}&\quad\text{if}\quad n>0\,;
\end{array}\right.
$$
see also the equality \eqref{pnast} below. The operators $a_n$ 
satisfy the relations 
\beq
\label{heisenberg-comm}
[\,a_{\ts m},a_{\ts n}\ts]=m\,\al\,\de_{m+n,0}\,.
\eeq
Thus $\La$ becomes a highest weight module~for 
an infinite-dimensional Heisenberg Lie algebra.
In terms of the operators $a_{\ts n}$ the Hamiltonian \Ref{H2} takes the form
$$
H^{(2)}=\sum_{m,n=1}^\infty
\bigl(\ts
a_{\ts-m}\ts a_{\ts-n}\ts a_{\ts m+n}+a_{\ts-m-n}\ts a_{\ts m}\ts a_{\ts n}
\ts\bigr)
+(\al-1)\sum_{n=1}^\infty n\,a_{\ts-n}\,a_{\ts n}\,.
$$
Similarly, the $H^{\ts(1)}_N$ yields a 
first order differential operator commuting with $H^{\ts(2)}$
\beq
\label{H1}
H^{(1)}=\lim_{N\rightarrow\infty} \al\ts H^{(1)}_N
=\,\sum_{n=1}^\infty\,\al\,n\,p_n\ts\frac{\dd}{\dd\ts p_n}\,=\,
\sum_{n=1}^\infty\,a_{\ts-n}\ts a_{\ts n}\,.
\eeq

%------------------------------------------------------------------------------

\subsection{Quantum field theory\/}

In terms of the field $\phi(s)$ on the circle $S^1\equiv\Rbbd/2\pi\Zbbd$
$$
\phi(s)= \sum_{n\neq0}\,a_n\exp\ts(-\ts\ri\,n\ts s)
\quad\text{for}\quad s\in S^1
$$
the commutation relations \Ref{heisenberg-comm} read
$$
[\ts\phi(s),\phi(t)\ts]=2\ts\pi\ts\ri\ts\al\ts\de^{\ts\prime}(s-t)\,.
$$

Setting $a_0=0$ and using the Wick normal ordering $:\ :$ with
the operators $a_n$ for $n<0$ to the left and for $n>0$ to the right,
the Hamiltonian $H^{(1)}$ becomes
$$
%H^{(1)}=
\frac{1}{2\ts\al} \sum_{m+n=0} :a_{\ts m}\ts a_{\ts n\!}:\,
\,=\,
\frac{1}{2\al}\,\int_{S^1}\frac{\rd s}{2\pi} \,:\ts\phi^2(s):
$$
while $H^{(2)}$ becomes the Hamiltonian of %the 
quantized Benjamin-Ono equation~\text{\cite{AW,Pol}}
\begin{align*}
\frac13\, 
\sum_{l+m+n=0}  :\ts a_{\ts l}\ts a_{\ts m}\ts a_{\ts n\!}: 
&\ +\ \ 
\frac{\al-1}{2}\sum_{m+n=0} \abs{n} :a_{\ts m}\ts a_{\ts n\!}:\ =
\\[2pt]
\int_{S^1}\frac{\rd s}{2\pi}\ 
\Bigl(\,\,\frac13\,
:\phi^3(s):\,
&\ + \ \ \frac{1-\al}{2}\,:\phi'(s)\,\bigl(\Hcal\ts\phi\bigr)(s):
\,\Bigr)
\end{align*}
where $\Hcal$ stands for the Hilbert transform
$$
\bigl(\Hcal\ts\phi\bigr)(s) 
=\text{p.v.} 
\int_{S^1}\frac{\rd t}{2\pi}\,
\cot\frac{s-t}{2}\,\,\phi(t)\,.
$$

In the particular case $\al=1$ the Jack symmetric polynomials degenerate 
into Schur polynomials. The Benjamin-Ono equation respectively degenerates 
into the dispersionless KdV (also called Burgers) equation.
An explicit construction of a countable set of commuting Hamiltonians for
the quantum dispersionless KdV
can be obtained via boson-fermion correspondence
and is available in terms of recurrence relations \cite{P} or 
a generating function \cite{OP,R}.

The higher quantum Hamiltonians for any parameter $\al$  
are constructed in
the present article, see the theorem in Subsection \ref{mainres}.
Note that in the case $\al=1$ our Hamiltonians $A^{\ts(k)}$ are different
from those considered in \cite{OP,P,R} being 
rather their polynomial combinations,
see for example \eqref{HS1} and \eqref{HS2} below. 

In their turn, Jack symmetric polynomials can be regarded as 
degenerations of the Macdonald polynomials 
in the variables $x_1\lcd x_N$ 
also depending on two formal parameters $q$ and $t\,$. 
The Jack case corresponds to $q=t^{\ts \al}$ where $t\to1\,$.
The Sekiguchi\ts-Debiard operators can be then regarded as
degenerations of the Macdonald operators \cite{M}
acting on the symmetric polynomials in %the variables 
$x_1\lcd x_N\ts$. 
Our theorem generalizes to the Macdonald case \cite{NS},
see also the earlier works~\text{\cite{AK,Shi}.}

%------------------------------------------------------------------------------

\subsection{Plan of the article}

In the next section %\ref{section::symfun} 
we recall some basic facts from the theory of symmetric functions
and set up the notation. We then introduce the Jack polynomials and 
Sekiguchi\ts-Debiard differential operators. Our main %technical 
tool is the notion of the symbol of an operator 
relative to the reproducing kernel for Jack polynomials. 
After establishing the basics, we state our main result 
which is an explicit formula for the symbol of the generating function
of commuting Hamiltonians. Then we explicitly construct
our elementary step operators for the Jack symmetric functions.
We finish Section \ref{section::symfun} with
reducing the proof of our theorem to certain determinantal identities
which are then proved in Section~\ref{section::Determinantal identities}.

In this article we generally keep to
the notation of the book \cite{M} for
symmetric functions. When using the results from \cite{M}
we will simply indicate their numbers within the book.
For example, the statement (1.11) from Chapter~I of the book
will be referred to as [I.1.11] assuming it is from~\cite{M}. 
We do not number our own
lemmas, propositions, theorems or corollaries
because we have only one of each.

%==============================================================================

\section{Symmetric functions}
\label{section::symfun}

%------------------------------------------------------------------------------

\subsection{Monomial functions\/}
\label{subone}

Fix any field $\FF\,$. For any positive integer $N$ 
denote by $\La_N$ the $\FF$-algebra of symmetric 
polynomials in $N$ variables $x_1\lcd x_N\,$.
The algebra $\La_N$ is graded by the polynomial degree.
The substitution $x_N=0$ defines
a homomorphism $\La_N\to\La_{N-1}$ preserving the degree.
Here $\La_{\ts0}=\FF\,$.
The inverse limit of the sequence
$$
\La_1\leftarrow\La_2\leftarrow\ldots
$$ 
in the category of graded algebras is denoted by $\La\,$. 
The elements of 
$\La$ are called \textit{symmetric functions\/}.    
Following \cite{M} we will introduce some standard bases of $\La\,$.

Let $\la=(\,\la_1,\la_2,\ldots\,)$ be any partition of $\,0,1,2,\ldots\,\,$. 
The number of non-zero parts is called the {\it length\/} of 
$\la$ and is denoted by $\ell(\la)\,$. 
If $\ell(\la)\le N$ then
the sum of all distinct 
monomials obtained by permuting the $N$ variables in
$x_1^{\,\la_1}\ldots x_N^{\,\la_N}$  is denoted by
$m_\la(x_1\lcd x_N)\,$. 
The symmetric polynomials $m_\la(x_1\lcd x_N)$ with
$\ell(\la)\le N$ form a basis of the vector space $\La_N\,$. 
By definition, for $\ell(\la)\le N$
\begin{equation}
\label{mon}
m_\la(x_1\lcd x_N)=
\sum_{1\le i_1<\ldots<i_k\le N}
\ \sum_{\si\in\Sg_k}\ \cc_\la^{\,-1}\ 
x_{i_{\si(1)}}^{\,\la_1}\ldots x_{i_{\si(k)}}^{\,\la_k}
\end{equation}
where we write $k$ instead of $\ell(\la)\,$. Here
$\Sg_k$ is the symmetric group permuting the numbers $1\lcd k$
and
\begin{equation}
\label{cela}
\cc_\la=k_1!\,k_2!\,\ldots
\end{equation}
if $k_1,k_2,\ldots$ are the respective multiplicites of the parts $1,2,\ldots$
of $\la\,$. Further,
\begin{equation}
\label{monst}
m_\la(x_1\lcd x_{N-1},0)=
\left\{
\begin{array}{cl}
m_\la(x_1\lcd x_{N-1})
&\quad\textrm{if}\quad\,\ell(\la)<N\,;
\\[2pt]
0
&\quad\textrm{if}\quad\,\ell(\la)=N\,.
\end{array}
\right.
\end{equation}
Hence for any fixed partition $\la$ the sequence of polynomials
$m_\la(x_1\lcd x_N)$ with $N\ge\ell(\la)$ 
has a limit in $\La\,$. This limit is called
the \textit{monomial symmetric function\/}
corresponding to $\la\,$. Simply omitting the variables,
we will denote the limit by $m_\la\,$.
With $\la$ ranging over all partitions of $0,1,2\ldots$ 
the symmetric functions $m_\la$ form a basis of the vector space $\La\,$.
Note that if $\ell(\la)=0$ then we set $m_\la=1\,$.

%------------------------------------------------------------------------------

\subsection{Power sums\/}
\label{subtwo}

For each $n=1,2,\ldots$ denote $p_n(x_1\lcd x_N)=x_1^n+\ldots+x_N^n\,$.
When the index $n$ is fixed
the sequence of symmetric polynomials $p_n(x_1\lcd x_N)$ with
$N=1,2,\ldots$ has a limit in $\La\,$, called
the \textit{power sum symmetric function} of degree $n\,$.
We denote the limit by $p_n\,$. 
More generally, for any partition $\la$ put
\begin{equation}
\label{pla}
p_\la=p_{\la_1}\,p_{\la_2}\ldots
\end{equation}
where we set $p_0=1\,$. The elements $p_\la$
form another basis of $\La\,$. In other words,
the elements $p_1,p_2,\ldots$ are
free generators of the commutative algebra $\La$ over $\FF\,$.

The basis of $p_\la$ can be related to the basis of
monomial symmetric functions as follows.
For any two partitions $\la$ and $\mu$
%of the same number $n$
denote by $R_{\ts\la\mu}$ the number of mappings 
$\theta:\{1\lcd\ell(\mu)\}\to\{1,2,\ldots\,\}$ such that
\begin{equation}
\label{refine}
\sum_{\theta(j)=i}\mu_j\,=\,\la_i
\quad\text{for each}\quad i=1,2,\ldots\,.
\end{equation}
For any such $\theta$ the partition $\mu$ in \eqref{refine}
is called a \textit{refinement} of $\la\,$.
Note that if $R_{\ts\la\mu}\neq0$ then $\la$ and $\mu$ are partitions
of the same number. Moreover, then by [I.6.10] we have
$\mu\le\la$ in the \textit{natural partial ordering\/} of partitions:
$$
\mu_1\le\la_1\,,\ \,
\mu_1+\mu_2\le\la_1+\la_2\,,\ \,
\ldots\ \,.
$$
By [I.6.9] we have
\begin{equation}
\label{pm}
p_\mu=\sum_\mu\,R_{\ts\la\mu}\,m_\la\,.
\end{equation}

%------------------------------------------------------------------------------

\subsection{Jack functions\/}
\label{subjack}

Now let $\FF$ be the field $\QQ(\al)$
where $\al$ is another variable. Define a bilinear form $\langle\ ,\,\rangle$ 
on the vector space $\La$ by setting for any $\la$ and $\mu$
\begin{equation}
\label{jackprod}
\langle\,p_\la,p_\mu\ts\rangle=\al^{\ell(\la)}\ts z_\la\,\de_{\la\mu}
\end{equation}
where
$$
z_\la=1^{k_1}k_1!\,2^{k_2}k_2!\ts\,\ldots
$$
in the notation \eqref{cela}. This form is obviously symmetric
and non-degenerate. By [Ex.\,VI.4.2]
there exists a unique family of elements $P_\la\in\La$ such that
$$
\langle\,P_\la,P_\mu\ts\rangle=0
\quad\text{for}\quad
\la\neq\mu
$$
and such that any $P_\la$ equals $m_\la$ 
plus a linear combination of the elements $m_\mu$
with $\mu<\la$ in the natural partial ordering.
The elements $P_\la\in\La$ are called 
the \textit{Jack symmetric functions\/}.
Alternatively, they can be defined as follows.

Denote by $\De(x_1\lcd x_N)$ the {\it Vandermonde polynomial\/} 
of $N$ variables
$$
\det\Big[x_i^{\,N-j}\Big]{\phantom{\big[}\!\!}_{i,j=1}^N=
\prod_{1\le i<j\le N}(x_i-x_j)\,.
$$
Put
\begin{equation}
\label{dnu}
S_N(u)=
\De(x_1\lcd x_N)^{-1}\cdot
\det\Big[\,
x_i^{\,N-j}\bigl(\ts u+j-1-\al\,x_i\ts\dd_i\ts\bigr)
\Big]{\phantom{\big[}\!\!}_{i,j=1}^N
\hspace{-8pt}
\end{equation}
where $u$ is a variable and $\dd_i$ is the operator
of partial derivation relative to $x_i\,$. 
Here the determinant is defined as the alternated sum
\begin{equation}
\label{sedeb}
\sum_{\si\in\Sg_N}
(-1)^{\si}\,
\prod_{i=1}^N\,\,\bigl(\,
x_i^{\,N-\si(i)}\bigl(\ts u+\si(i)-1-\al\,x_i\ts\dd_i\ts\bigr)\bigr)
\end{equation}
where as usual $(-1)^{\si}$ denotes the sign of permutation $\si\,$.
In every product over $i=1\lcd N$ appearing in \eqref{sedeb}
the operator factors pairwise commute, 
hence their ordering does not matter. Further,
$S_N(u)$ is a polynomial in the variable $u$ with pairwise
commuting operator coefficients
preserving the space $\La_N\,$, see for instance [Ex.\,VI.3.1].
We will call the restrictions of these
coefficients to the space $\La_N$ the 
\textit{Sekiguchi\ts-Debiard operators\/}. 
By [Ex.\,VI.4.2] the latter operators
have a common eigenbasis in $\La_N$ parametrized by
partitions $\la$ of length $\ell(\la)\le N\ts$. 
The eigenvectors are called the 
\textit{Jack symmetric polynomials\/}. 

For each $\la$ with $\ell(\la)\le N$
there is an eigenvector denoted by $P_\la(x_1\lcd x_N)$
which is equal to $m_\la(x_1\lcd x_N)$ plus
a linear combination of the polynomials 
$m_\mu(x_1\lcd x_N)$ with $\mu<\la\,$ and $\ell(\mu)\le N\ts$. 
It turns out that each coefficient
in this linear combination does not depend on $N\ts$.
Note that if $\la$ and $\mu$ are any two partitions of the same number
such that $\la\ge\mu\,$, then $\ell(\la)\le\ell(\mu)$ due to
[I.1.11]. It follows that the polynomials
$P_\la(x_1\lcd x_N)$
enjoy the same \textit{stability property} as the 
polynomials $m_\la(x_1\lcd x_N)$ in \eqref{monst}:
\begin{equation}
\label{jackst}
P_\la(x_1\lcd x_{N-1},0)=
\left\{
\begin{array}{cl}
P_\la(x_1\lcd x_{N-1})
&\quad\textrm{if}\quad\,\ell(\la)<N\,;
\\[2pt]
0
&\quad\textrm{if}\quad\,\ell(\la)=N\,.
\end{array}
\right.
\end{equation}
In particular, the sequence of polynomials
$P_\la(x_1\lcd x_N)$ with $N\ge\ell(\la)$
has a limit in $\La\,$. This is exactly
the Jack symmetric function $P_\la\,$.
The eigenvalues of Sekiguchi\ts-Debiard operators acting on $\La_N$ 
are also known. By [Ex.\,VI.4.2]  
\begin{equation}
\label{deigen}
S_N(u)\,P_\la(x_1\lcd x_N)\,=\,
\prod_{i=1}^N\,\bigl(\ts u+i-1-\al\,\la_i\ts\bigr)
\cdot P_\la(x_1\lcd x_N)\,.
\end{equation}

%------------------------------------------------------------------------------

\subsection{Reproducing kernel\/}
\label{reproker}

In this subsection we will regard the elements of $\La$
as infinite sums of finite products of the variables
$x_1,x_2,\ldots\,\,$. For instance, we have
$$
p_n=x_1^n+x_2^n+\ldots
$$
for any $n\ge1\,$. When we need to distinguish 
$x_1,x_2,\ldots$ from any other variables,
we will write $f(\ts x_1,x_2,\ldots\ts)$ instead of any $f\in\La\,$.
Now let $y_1,y_2,\ldots$ be
variables independent of $x_1,x_2,\ldots\,\,$. 
According to [VI.10.4]
with the bilinear form \eqref{jackprod} % on $\La$
one associates the \textit{reproducing kernel}
\begin{equation}
\label{kernel}
\Pi\,=\prod_{i,j=1}^\infty 
(\ts1-x_i\ts y_j)^{-1/\al}.
\end{equation}
This $\Pi$ should be regarded as an infinite sum of
monomials in $x_1,x_2,\ldots$ and in $y_1,y_2,\ldots$ by expanding 
the factor corresponding to $i,j$ as a series at $x_i\ts y_j\to0\,$.
 
The property of $\Pi$ most useful for us can be stated as 
the following lemma.
For any $f\in\La$
denote by $f^{\ts\ast}$ the operator on $\La$
adjoint to the multiplication by $f$
relative to the bilinear form \eqref{jackprod}. 
Note that here $f=f(\ts x_1,x_2,\ldots\ts)\,$.

\begin{lem}
We have
\begin{equation}
\label{useful}
f^{\ts\ast}(\Pi)/\Pi=f(\ts y_1,y_2,\ldots\ts)\,.
\end{equation}
\end{lem}

\begin{proof}
The commutative algebra $\La$ is 
generated by the elements $p_n$ with $n\ge1\,$.
Therefore it suffices to prove \eqref{useful} for $f=p_n$ only.
Consider the operator $\dd/\dd\ts p_n$ of derivation 
in $\La$ relative to $p_n=p_n(\ts x_1,x_2,\ldots\ts)\,$.
By the definition \eqref{jackprod} we have
\begin{equation}
\label{pnast}
p_n^{\ts\ast}=\al\,n\,\dd/\dd\ts p_n
\end{equation}
On the other hand, by taking the logarithm of \eqref{kernel} 
and then exponentiating,
$$
\Pi\,=\,\exp\,\Bigl(\,\,
\sum_{n=1}^\infty\,
{p_n(\ts x_1,x_2,\ldots\ts)\,p_n(\ts y_1,y_2,\ldots\ts)}/{\al\,n}
\,\Bigr)\ts.
$$
The relation \eqref{useful} for $f=p_n$ follows from
the last two displayed equalities.
\qed
\end{proof}

%------------------------------------------------------------------------------

\subsection{Main result\/}
\label{mainres}

Let $\FF=\QQ(\al)$ as in the previous two subsections.
For $N\ge1$ let
$\rho_N$ be the homomorphism $\La_N\to\La_{N-1}$
defined by setting $x_N=0\,$, as
in the beginning of Subsection \ref{subone}. Denote
\begin{equation}
\label{snu}
A_N(u)=S_N(u)/(u)_N
\end{equation}
where we employ the \textit{Pochhammer symbol}
$$
(u)_N=u\,(u+1)\ldots(u+N-1)\,.
$$
The right hand side of the equation \eqref{snu} is regarded as
a rational function of $u$ with the values
being operators acting on the space $\La_N\,$.
Due to the stability property \eqref{jackst}
of Jack symmetric polynomials,
the equation \eqref{deigen} implies~that
$$
\rho_N\,A_N(u)=A_{N-1}(u)\,\rho_N
$$
where $A_{\ts0}(u)=1\,$.
So the sequence of $A_N(u)$ with $N\ge1$ has a limit
at $N\to\infty$. This limit can be written as a series
$$
A(u)=1+A^{\ts(1)}/(u)_1+A^{\ts(2)}/(u)_2+\ldots
$$
where the 
inverses of the $(u)_1 ,(u)_2 ,\,\ldots$ can be regarded as series in $u^{-1}$  
whereas $A^{\ts(1)},A^{\ts(2)},\,\ldots$
are certain linear operators acting on
$\La\,$. By definition, the Jack symmetric functions are
joint eigenvectors of these operators. In particular,
the operators $A^{\ts(1)},A^{\ts(2)},\,\ldots$ 
pairwise commute, and are self-adjoint
relative to the bilinear form \eqref{jackprod}.
We call them 
the \textit{Sekiguchi\ts-Debiard operators at infinity\/}.
Due to the property \eqref{jackst} their definition
immediately implies that
$$
A^{\ts(k)}P_\la=0
\quad\text{if}\quad
\ell(\la)<k\,.
$$

It is also transparent from \eqref{deigen} that 
for any homogeneous $f\in\La$
$$
A^{(1)}f=-\,\al\,\deg f\,.
$$  
Hence in the notation \eqref{H1}
\beq
\label{HS1}
-\,A^{\ts(1)}=H^{\ts(1)}
\ts.
\eeq
The operator $A^{(2)}$ is well studied
[Ex.\,VI.4.3]. In particular, it is known~that
\beq
\label{HS2}
A^{\ts(1)}(\ts A^{\ts(1)}\ts\!+1)-2\,A^{\ts(2)}=
H^{\ts(2)}
\eeq
in the notation \eqref{H2}.
The main result of our article is the more general

\begin{teo}
In the notation \eqref{cela}
for each\/ $k=1,2,\ldots$ we have
\begin{equation}
\label{basic}
A^{\ts(k)}=(-1)^k\sum_{\ell(\la)=k}
\cc_\la\,m_\la\,m_\la^{\ts\ast}
\end{equation}
where $\la$ ranges over all partitions of length $k\,$.
\end{teo}

By inverting the relation \eqref{pm} any monomial symmetric function
$m_\la$ can be expressed as a linear combination of the functions
$p_\mu$ where $\la\,,\mu$ are partitions of the same number
and $\la\le\mu\,$. By substituting into \eqref{basic}
and using \eqref{pla},\eqref{pnast}
one can write each operator $A^{\ts(k)}$
in terms of $p_n$ and $\dd/\dd\ts p_{n}$ where $n=1,2,\ldots\,\,$.
In particular, one recovers the above formulas 
for the operators $A^{\ts(1)}$ and $A^{\ts(2)}$.

%------------------------------------------------------------------------------

\subsection{Step operators\/}
\label{shiftop}

In this subsection we will get a corollary to our theorem
by using the following particular case of the {\it Pieri rule}
for Jack symmetric functions. By [VI.6.24] for any partition $\mu$
the product $p_1\ts P_\mu$ equals the linear combination of
the symmetric functions $P_\la$ with the coefficients
\begin{equation}
\label{bml}
\prod_{j=1}^{i-1}\,
\frac{\,\al\ts(\la_i-\la_j)-i+j-1\,}{\al\ts(\la_i-\la_j-1)-i+j}
\,\cdot\,
\prod_{j=1}^{i-1}\,
\frac{\,\al\ts(\la_i-\la_j-1)-i+j+1\,}{\al\ts(\la_i-\la_j)-i+j}
\end{equation}
where $\la$ ranges over all partitions such that 
the sequence $\la_1,\la_2,\ldots$ is obtained from $\mu_1,\mu_2,\ldots$
by increasing one of its terms by $1$ and $i$ is
the index of the term.
 
Further, by [VI.6.19] the above stated equality implies that for any 
partition $\la$ the symmetric function 
$\dd\ts P_\la/\dd\,p_1=\al^{-1}p_1^{\ast}\ts P_\la$ 
equals the linear combination of the $P_\mu$ with the coefficients
\begin{equation}
\label{clm}
\prod_{j=1}^{\la_i-1}
\frac{\,\al\ts(\la_i-j-1)+\lap_j-i+1\,}{\al\ts(\la_i-j)+\lap_j-i}
\,\,\cdot
\prod_{j=1}^{\la_i-1}
\frac{\al\ts(\la_i-j+1)+\lap_j-i}{\,\al\ts(\la_i-j)+\lap_j-i+1\,}
\end{equation}
where $\mu$ ranges over all partitions such that 
the sequence $\mu_1,\mu_2,\ldots$ is obtained from $\la_1,\la_2,\ldots$
by decreasing one of its terms by $1$ and $i$ is
the index of the term. As usual, here $\lap=(\ts\lap_1,\lap_2,\ldots\ts)$ 
is the partition conjugate to $\la\ts$. 
 
Now define the linear operators 
$B^{\ts(1)},B^{\ts(2)},\,\ldots$ acting on $\La$
by setting
$$
B(u)=B^{\ts(1)}/(u)_1+B^{\ts(2)}/(u)_2+\ldots
$$
where
\begin{equation}
\label{svbu}
[\,p_1,A(u)\ts]=\al\,B(u)
\end{equation}
while the square brackets denote the operator commutator.
Further, define the operators 
$C^{\ts(1)},C^{\ts(2)},\,\ldots$ acting on $\La$
by setting
$$
C(u)=C^{\ts(1)}/(u)_1+C^{\ts(2)}/(u)_2+\ldots
$$
where
\begin{equation}
\label{svcu}
[\ts A(u),\dd/\dd\ts p_1\ts]=C(u)\ts.
\end{equation}
Our definitions of the operators
$B^{\ts(1)},B^{\ts(2)},\,\ldots$ and 
$C^{\ts(1)},C^{\ts(2)},\,\ldots$ are
motivated by the results of \cite{SV}.
Our theorem yields explicit expressions for these operators, 
stated as the following corollary.
The corollary will then allow us
to construct the elementary step 
operators  for Jack symmetric functions, see
\eqref{bshift} and \eqref{cshift}.

\begin{cor}
For every $k=0,1,2,\ldots$ we have the equalities
\begin{align}
\label{boper}
B^{\ts(k+1)}&=(-1)^k\sum_{\ell(\mu)=k}
\cc_{\ts\mu\sqcup1}\,m_{\ts\mu\sqcup{1}}\,m_\mu^{\ts\ast}\,,
\\[4pt]
\label{coper}
C^{\ts(k+1)}&=(-1)^k\sum_{\ell(\mu)=k}
\cc_{\ts\mu\sqcup1}\,m_\mu\,m_{\ts\mu\sqcup{1}}^{\ts\ast}
\end{align}
where $\mu\sqcup1$ denotes the partition obtained from $\mu$
by appending one extra part~$1\ts$. 
\end{cor}

\begin{proof}
The equalities \eqref{boper} and \eqref{coper}
follow from each other, because by \eqref{pnast}
$$
B(u)^{\ts\ast}=C(u)\,.
$$ 
But the equality \eqref{coper} follows from \eqref{basic}
and \eqref{svcu} by using [Ex.\,I.5.3].
\qed
\end{proof}

By the definition \eqref{snu} of the series $A(u)$, for any partition $\la$ 
we~have
\begin{equation}
\label{ajack}
A(u)\,P_\la\,=\,
\prod_{i=1}^{\infty}\,\frac{u+i-1-\al\,\la_i}{u+i-1}
\cdot P_\la\,,
\end{equation}
see \eqref{deigen}. In the infinite product 
displayed above the only factors different from $1$ are those
corresponding to $i=1\lcd\ell(\la)\ts$.
For any such $i$ consider the product
\begin{equation}
\label{iskip}
\frac1{u+i-1}\,\ts
\prod_{\substack{j=1\\j\neq\ts i\ts}}^{\ell(\la)}\,
\frac{u+j-1-\al\,\la_j}{u+j-1}\ .
\end{equation}

It follows from the definition \eqref{svbu}
that for any given partition $\mu$ we have %the equality
\begin{equation}
\label{parti}
B(u)\,P_\mu=\sum_{\la}\,B_{\ts\la\mu}(u)\,P_\la
\end{equation}
where $B_{\ts\la\mu}(u)$ is the product of \eqref{bml} by \eqref{iskip}
while $\la$ ranges over all partitions such that 
the sequence $\la_1,\la_2,\ldots$ is obtained from $\mu_1,\mu_2,\ldots$
by increasing one of its terms by $1$ and $i$ is
the index of the term. Similarly, it follows 
from the definition \eqref{svcu} that
for any given partition $\la$ we have
\begin{equation}
\label{partii}
C(u)\,P_\la=\sum_{\mu}\,C_{\mu\la}(u)\,P_\mu
\end{equation}
where $C_{\mu\la}(u)$ is the product of \eqref{clm} by \eqref{iskip} and by $\al$
while $\mu$ ranges over all partitions such that 
the sequence $\mu_1,\mu_2,\ldots$ is obtained from $\la_1,\la_2,\ldots$
by decreasing one~of its terms by $1$ and $i$ is
the index of the term. 

Let the partition $\la$ be fixed. 
Then for $i=1\lcd\ell(\la)$ the elements $\al\ts\la_i-i+1$ 
of the field $\QQ(\al)$ are pairwise distinct.
Therefore by the part (i) of the corollary
for the partition $\mu$ corresponding to any of these indices $i$ we have
\begin{equation}
\label{bshift}
B(\ts\al\ts\la_i-i+1\ts)\,P_\mu=B_{\ts\la\mu}(\ts\al\ts\la_i-i+1\ts)\,P_\la
\end{equation}
where the coefficient $B_{\ts\la\mu}(\ts\al\ts\la_i-i+1\ts)$
is the product of \eqref{bml} by
\begin{equation}
\label{iskipla}
\prod_{j=1}^{\ell(\la)}\ts\frac1{\al\ts\la_i-i+j}
\,\,\cdot\,
\prod_{\substack{j=1\\j\neq i}}^{\ell(\la)}\ts
{(\ts\al\ts\la_i-\al\ts\la_j-i+j\ts)}\,.
\end{equation}
The left hand side of the equality \eqref{bshift} should be understood
as the value in $\La$ of the rational function $B(u)\,P_\mu$ at the point
$u=\al\ts\la_i-i+1\ts$. Similarly, by (ii)
\begin{equation}
\label{cshift}
C(\ts\al\ts\la_i-i+1\ts)\,P_\la=C_{\mu\la}(\ts\al\ts\la_i-i+1\ts)\,P_\mu
\end{equation}
where 
$C_{\mu\la}(\ts\al\ts\la_i-i+1\ts)$
is the product of \eqref{clm} by \eqref{iskipla} and by
$\al\ts$.

%------------------------------------------------------------------------------

\subsection{Reduction of the proof}

In this subsection we reduce the proof of our theorem
to proving a certain determinantal identity
for each $N=1,2,\ldots\,\,$.
This identity will be proved
in the next section by using the induction on $N\ts$. 

By the lemma from Subsection \ref{reproker}
our theorem is equivalent to the equality
\begin{equation}
\label{syminf}
A(u)(\Pi)/\Pi\,=\,
\sum_{k=0}^\infty\ 
\frac{(-1)^k}{(u)_k}
\sum_{\ell(\la)=k}\,
\cc_\la\,m_\la(\ts x_1,x_2,\ldots\ts)\,m_\la(\ts y_1,y_2,\ldots\ts)
\end{equation} 
where the coefficients of the series 
$A(u)$ are regarded as operators acting on the
symmetric functions in the variables $x_1,x_2,\ldots\ $.
Here we set $(u)_{\ts0}=1\,$.
It suffices to prove for each $N=1,2,\ldots$
the restriction of the functional equality \eqref{syminf} to 
\begin {equation}
\label{xrest}
x_{N+1}=x_{N+2}=\ldots=0\,.
\end{equation}

By the very definition of $A(u)$
the restriction of the left hand side of \eqref{syminf}
to \eqref{xrest}  as of a function in the variables 
$x_1,x_2,\ldots$ equals 
\begin{equation}
\label{pmsp}
A_N(u)(\Pi_N)/\Pi_N
\end{equation}
where we denote
$$
\Pi_N\,=\,
\prod_{i=1}^N\, 
\prod_{j=1}^\infty\,\ts 
(\ts1-x_i\ts y_j)^{-1/\al}.
$$ 
Simply by the definition of the monomial symmetric function 
$m_\la(\ts x_1,x_2,\ldots\ts)$ its restriction to \eqref{xrest}
is $m_\la(\ts x_1\lcd x_N)$ if $\ell(\la)\le N$
and vanishes if $\ell(\la)>N\ts$. 
Therefore the restriction of the right hand side of
\eqref{syminf} to \eqref{xrest} equals
\begin{equation}
\label{symfininf}
\sum_{k=0}^N\ 
\frac{(-1)^k}{(u)_k}
\sum_{\ell(\la)=k}\,
\cc_\la\,m_\la(\ts x_1\lcd x_N)\,m_\la(\ts y_1,y_2,\ldots\ts)\,.
\end{equation} 
Further, due to [VI.2.19]
to prove the equality of \eqref{pmsp} to \eqref{symfininf}
it suffices~to~set 
$$
y_{N+1}=y_{N+2}=\ldots=0\,.
$$
However, we will keep working with the infinite collection of variables
$y_1,y_2,\ldots\,\,$.
This will simplify the induction argument in the next section. 

Let us compute the function \eqref{pmsp}.
It depends on the variable $u$ rationally. It is also symmetric in 
either of the two collections of variables
$x_1\lcd x_N$ and $y_1,y_2,\ldots\,\,$. This function can be obtained
by applying to the identity function $1$
the result of conjugating $A_N(u)$ by the operator of
multiplication by $\Pi_N\ts$. 

Conjugating the operator \eqref{sedeb}
by $\Pi_N$ amounts to replacing
each $\dd_i$ in \eqref{sedeb} with the sum 
$$
\dd_i+\al^{-1}\,
\sum_{l=1}^\infty\,\ts\frac{y_l}{1-x_i\ts y_l}\,.
$$

\medskip\noindent
Here we are just adding to each $\dd_i$ the logarithmic derivative
of the function~$\Pi_N$ relative to $x_i\,$. 
Hence conjugating \eqref{sedeb} by the multiplication by 
$\Pi_N$ yields 
\begin{equation*}
\label{conjug}
\sum_{\si\in\Sg_N}
(-1)^{\si}\,
\prod_{i=1}^N\,\,\Bigl(\,
x_i^{\,N-\si(i)}\Bigl(\ts 
u+\si(i)-1-\al\,x_i\ts\dd_i
-\sum_{l=1}^\infty\,\ts\frac{x_i\ts y_l}{1-x_i\ts y_l}\,
\Bigr)\Bigr)\,.
\end{equation*}
Here in any single summand
each of the factors corresponding to
$i=1\lcd N$ does not depend on the variables $x_j$ with $j\neq i\,$.
Therefore when applying the latter operator sum
to the identity function $1$ we can simply replace 
each $\dd_i$ with zero. Then we get the function
\begin{gather}
\notag
\sum_{\si\in\Sg_N}
(-1)^{\si}\,
\prod_{i=1}^N\,\,\Bigl(\,
x_i^{\,N-\si(i)}\Bigl(\ts 
u+\si(i)-1-\sum_{l=1}^\infty\,\ts\frac{x_i\ts y_l}{1-x_i\ts y_l}\,
\Bigr)\Bigr)\ =
\\[4pt]
\label{symsym}
\det\Big[\,x_i^{\,N-j}\Big(u+j-1+\,
\sum_{l=1}^\infty\,\frac{x_i\ts y_l}{x_i\ts y_l-1}\,\Big)\,\Big]
{\phantom{\big[}\!\!}_{i,j=1}^N\,.
\end{gather}

It follows that the function \eqref{pmsp}
is equal to the determinant \eqref{symsym}
divided by the Vandermonde polynomial $\De(x_1\lcd x_N)$
and by the Pochhammer symbol $(u)_N\ts$, see %the definitions 
\eqref{dnu} and \eqref{snu}.
This ratio is equal to the right hand side of
\eqref{symfininf} by the proposition in Subsection \ref{mainid},
see the argument at the end of that subsection.
We will prove the proposition in Subsections \ref{expdet} to \ref{cancel}.
Thus we will complete the proof of our~theorem.
Note that another proof of this theorem can be obtained
by using the results of \cite[Sec.\,3]{AK} and \cite[Sec.\,9]{Shi}
on the Macdonald operators.

%==============================================================================

\section{Determinantal identities}
\label{section::Determinantal identities}

%------------------------------------------------------------------------------
 
\subsection{Getting the theorem\/}
\label{mainid}

Let $\FF$ be any field. 
Consider the rational function of two variables $u,v$
\begin{equation}
\label{ka}
\Psi(u,v)=\frac{u\,v}{u\,v-1}
\end{equation}
with values in $\FF\,$. This function is a solution of the equation 
\begin{equation}
\label{kk}
(u-v)\,\Psi(u,w)\,\Psi(v,w)=u\,\Psi(v,w)-v\,\Psi(u,w)\,.
\end{equation}
Here $w$ is a third variable. One can easily demonstrate that 
any non-zero rational solution $\Psi(u,v)$ 
of \eqref{kk} has the form
$
{u}\ts/(u-\psi(v))
$
where $\psi(v)$ is an arbitrary rational function of a single variable.
In particular, by choosing $\psi(v)=1/v$ we get the solution 
\eqref{ka}.

Let $x_1\lcd x_N$ and  $y_1,y_2,\ldots$ be independent variables.
Here we assume that $N\ge1\,$. 
In the next three subsections we will prove 
the following proposition. 
%We will use methods already employed in \cite{KNP}.

\begin{pro}
For any solution\/ $\Psi(u,v)$ of the equation \eqref{kk} 
we have an \text{identity}
\begin{gather}
\nonumber
\det\Big[\,x_i^{\,N-j}\Big(u+j-1+\,\sum_{l=1}^\infty\,\Psi(x_i,y_l)\Big)\Big]
{\phantom{\big[}\!\!}_{i,j=1}^N=
\\[2pt]
\label{8}
\De(x_1\lcd x_N)
\ \sum_{k=0}^N\ 
(u+k)\ldots(u+N-1)
\ \sum_{\substack{i_1\lcd i_k\\j_1\lcd j_k}}
\ \prod_{r=1}^k\ \ \Psi(x_{i_r},y_{j_r})
\end{gather}
where all the indices\/ $i_1\lcd i_k\in\{1\lcd N\}$ are 
distinct, the indices\/ $j_1\lcd j_k\in\{1,2,\ldots\,\}$ are 
all distinct too,
and the sum is taken over all collections of these indices such that\/
{\rm different} are all the corresponding sets of\/ $k$ pairs 
\begin{equation}
\label{ia}
\big\{(i_1,j_1)\lcd(i_k,j_k)\big\}\,.
\end{equation}
\end{pro}

\medskip
For any $k\ge1$ take the symmetric group $\Sg_k\,$.
Using the permutations $\si\in\Sg_k$ the sum over the indices
$i_1\lcd i_k$ and $j_1\lcd j_k$ 
at the right hand side of the equality \eqref{8} can be also written as
\begin{equation}
\label{sis}
\sum_{\substack{i_1<\ldots<i_k\\j_1<\ldots<j_k}}
\ \sum_{\si\in\Sg_k}\ \ \prod_{r=1}^k\ \ \Psi(x_{i_r},y_{j_{\si(r)}})\,.
\end{equation}
Hence by choosing the function 
$\Psi(u,v)$ as in \eqref{ka} 
our proposition implies that the determinant 
\eqref{symsym} equals 
\begin{gather}
\notag
\De(x_1\lcd x_N)
\ \sum_{k=0}^N\ 
(u+k)\ldots(u+N-1)\ \times
\\[2pt]
\label{10}
\sum_{\ell(\la)=k}
\cc_\la\,m_\la(x_1\lcd x_N)\,m_\la(y_1,y_2,\ldots\,)\,.
\end{gather}

Indeed, if $\Psi(u,v)$ is the rational function \eqref{ka}
then the sum \eqref{sis} equals
\begin{gather*}
\sum_{\substack{i_1<\ldots<i_k\\j_1<\ldots<j_k}}
\ \sum_{\si\in\Sg_k}\ \ \prod_{r=1}^k\ \ 
\frac{x_{i_r}y_{j_{\si(r)}}}{x_{i_r}y_{j_{\si(r)}}-1}\ =
\\
\sum_{\substack{i_1<\ldots<i_k\\j_1<\ldots<j_k}}
\ \sum_{\si\in\Sg_k}\,
\sum_{l_1\lcd l_k=1}^\infty
\ \prod_{r=1}^k\ \ 
(x_{i_r}y_{j_{\si(r)}})^{\,l_r}\ =
\\
\sum_{\substack{i_1<\ldots<i_k\\j_1<\ldots<j_k}}
\ \sum_{\si,\tau\in\Sg_k}\,
\sum_{\ell(\la)=k}
\ \cc_\la^{\,-1}\ \,\prod_{r=1}^k\ \ 
(x_{i_r}y_{j_{\si(r)}})^{\,\la_{\tau(r)}}\,.
\end{gather*}

\noindent
Here $\la$ ranges over all partitions $\la$ of length $k\,$. 
The sum displayed in the last line equals 
the sum in the second line of \eqref{10}, by using the expression
\eqref{mon} for the
polynomial $m_\la(x_1\lcd x_N)$ and a similar expression for 
$m_\la(y_1,y_2,\ldots\,)$\,.~Thus our proposition implies 
the theorem as stated in Subsection \ref{mainres}.

%------------------------------------------------------------------------------

\subsection{Expanding the determinant\/}
\label{expdet}

We will prove the proposition by induction on $N$.
The case $N=1$ is the induction base. Here
the left hand side of 
\eqref{8}~is
$$
u+\,\sum_{l=1}^\infty\,\Psi(x_1,y_l)\,.
$$
The right hand side of \eqref{8} is then the same by definition, 
since~$\De(x_1)=1\,$.

Now take $N>1$ and assume that the identity \eqref{8}
holds for $N-1$ instead of~$N$. For each index $i=1\lcd N$ 
we will for short denote
$$
\De_{\ts i}=
\De(x_1\lcd\xh_i\lcd x_N)
$$
where as usual the symbol $\xh_i$ indicates the omitted variable.
By expanding the determinant 
at the left hand side of \eqref{8} in the first column 
and then using the induction assumption with $u+1$ instead of $u\,$, 
we get the sum 
\begin{gather}
\nonumber
\sum_{i=1}^N\ 
(-1)^{i+1}x_i^{\,N-1}\,u\,\,
\De_{\ts i}
\ \times
\\[4pt]
\nonumber
\ \sum_{k=0}^{N-1}\ 
(u+k+1)\ldots(u+N-1)
\ \sum_{\substack{i_1\lcd i_k\neq i\\j_1\lcd j_k}}\ \prod_{r=1}^k
\ \ \Psi(x_{i_r},y_{j_r})\ +
\\[4pt]
\nonumber
\sum_{i=1}^N\ 
(-1)^{i+1}x_i^{\,N-1}\ \sum_{l=1}^\infty\ \Psi(x_i,y_l)\ 
\De_{\ts i}
\ \times
\\[4pt]
\label{one}
\ \sum_{k=0}^{N-1}\ 
(u+k+1)\ldots(u+N-1)
\ \sum_{\substack{i_1\lcd i_k\neq i\\j_1\lcd j_k}}
\ \prod_{r=1}^k\ \ \Psi(x_{i_r},y_{j_r})\,.
\end{gather}
Now consider the sum 
$$
\sum_{i=1}^N\ 
(-1)^{i+1}x_i^{\,N-1}\,
\De_{\ts i}\ \ 
\sum_{l=1}^\infty\,
\sum_{\substack{i_1\lcd i_k\neq i\\j_1\lcd j_k}}
\Psi(x_i,y_l)
\ \prod_{r=1}^k\ \Psi(x_{i_r},y_{j_r})
$$
coming from the last two lines of the display \eqref{one}.
This sum can be written~as
\begin{gather}
\nonumber
\sum_{i=1}^N\ 
(-1)^{i+1}x_i^{\,N-1}\,
\De_{\ts i}
\sum_{\substack{i_1\lcd i_k\neq i\\j_1\lcd j_k}}\ 
\sum_{\substack{l\neq j_1\lcd j_k}}
\Psi(x_i,y_l)
\ \prod_{r=1}^k\ \Psi(x_{i_r},y_{j_r})\ +
\\[4pt]
\label{linabo}
\sum_{i=1}^N\ 
(-1)^{i+1}x_i^{\,N-1}\,
\De_{\ts i}
\sum_{\substack{i_1\lcd i_k\neq i\\j_1\lcd j_k}}
\,\sum_{s=1}^k\ 
\Psi(x_i,y_{j_s})
\ \prod_{r=1}^k\ \Psi(x_{i_r},y_{j_r})\,.
\end{gather}
The next two subsections
will show that the sum in the second line of~\eqref{linabo}~equals
\begin{equation}
\label{20}
k\ \ 
\sum_{i=1}^N\ 
(-1)^{i+1}x_i^{\,N-1}\,
\De_{\ts i}\ 
\sum_{\substack{i_1\lcd i_k\neq i\\j_1\lcd j_k}}\ \prod_{r=1}^k
\ \  \Psi(x_{i_r},y_{j_r})\,.
\end{equation}

Due to that equality the sum \eqref{one} can be rewritten as
\begin{gather}
\nonumber
\sum_{i=1}^N
(-1)^{i+1}x_i^{\,N-1}\,
\De_{\ts i}
\ \times
\\[10pt]
\nonumber
\Bigl(\,
\ \sum_{k=0}^{N-1}\ 
(u+k)\,(u+k+1)\ldots(u+N-1)
\ \sum_{\substack{i_1\lcd i_k\neq i\\j_1\lcd j_k}}\ \prod_{r=1}^k\ 
\Psi(x_{i_r},y_{j_r})\ +
\\[2pt]
\label{M-1k+1}
\ \sum_{k=0}^{N-1}\ 
(u+k+1)\ldots(u+N-1)
\sum_{\substack{i_1\lcd i_k\neq i\\j_1\lcd j_k\neq l}}
\Psi(x_i,y_l)
\ \prod_{r=1}^k\ \Psi(x_{i_r},y_{j_r})\Bigr)\,.
\end{gather}
Here the indices $i_1 , j_1\lcd i_k , j_k$ and $l$ are assumed to be running.
In the second line of \eqref{M-1k+1} 
we can include the index 
$k=N$ to the summation range
without affecting the sum since $k$ distinct indices 
$i_1\lcd i_k\neq i$ exist only if  $k<N\,$.
In the third line 
we can replace $k+1$ with $k$ where $k=1\lcd N$.~We~get 
$$
\ \sum_{k=1}^N\ 
(u+k)\ldots(u+N-1)
\sum_{\substack{i_1\lcd i_{k-1},i\\j_1\lcd j_{k-1},l}}
\Psi(x_i,y_l)
\ \prod_{r=1}^{k-1}\ \Psi(x_{i_r},y_{j_r})\,.
$$
Here the set 
$$
\big\{((i_1,j_1)\lcd(i_{k-1},j_{k-1}),(i,l)\big\}\!\!\!\!
$$
can be any of the sets \eqref{ia} appearing in \eqref{8},
provided that in \eqref{ia} one of the indices $i_1\lcd i_k$ coincides with
the given index $i\,$. That one index can be then denoted by $i_k$ 
because the order of
the $k$ elements of the set \eqref{ia} does~not~matter.

These observations will show that 
the sum displayed in the second and the third
lines of of \eqref{M-1k+1} equals
$$
\ \sum_{k=0}^N\ 
(u+k)\ldots(u+N-1)\ 
\sum_{\substack{i_1\lcd i_k\\j_1\lcd j_k}}\ \prod_{r=1}^k
\ \ \Psi(x_{i_r},y_{j_r})\,.
$$
In particular, they will show that this sum
does not depend on the index $i\,$. 
Hence we will have the identity \eqref{8} proved,
by expanding the determinant $\De(x_1\lcd x_N)$ 
in the first column.

%------------------------------------------------------------------------------

\subsection{Two sums\/}
\label{twosums}

We need to prove for $k=0\lcd N-1$ that the sum 
\eqref{20} equals the sum in the second line of \eqref{linabo}.  
By using \eqref{kk} the last mentioned sum can be written~as 
\begin{gather*}
\sum_{i=1}^N\ 
\sum_{\substack{i_1\lcd i_k\neq i\\j_1\lcd j_k}}
\,\sum_{s=1}^k\ 
(-1)^{i+1}x_i^{\,N-1}\,
\De_{\ts i}\ \times
\\[4pt]
\nonumber
\frac{x_i\,\Psi(x_{i_s},y_{j_s})-x_{i_s}\,\Psi(x_i,y_{j_s})}{x_i-x_{i_s}}
\ \,\prod_{r\neq s}\ \Psi(x_{i_r},y_{j_r})
\end{gather*}
which by subtracting the product $x_{i_s}\,\Psi(x_{i_s},y_{j_s})$
from the numerator of the above displayed fraction and then adding it back
can be rewritten as
\begin{gather}
\nonumber
\sum_{i=1}^N\ 
\sum_{\substack{i_1\lcd i_k\neq i\\j_1\lcd j_k}}
\,\sum_{s=1}^k\ 
(-1)^{i+1}x_i^{\,N-1}\,
\De_{\ts i}
\ \,\prod_{r=1}^k\ \Psi(x_{i_r},y_{j_r})\ +
\\[2pt]
\nonumber
\sum_{i=1}^N\ 
\sum_{\substack{i_1\lcd i_k\neq i\\j_1\lcd j_k}}
\,\sum_{s=1}^k\ 
(-1)^{i+1}x_i^{\,N-1}\,
\De_{\ts i}\ \times
\\[4pt]
\label{33}
\biggl(\,
\frac{x_{i_s}\,\Psi(x_{i_s},y_{j_s})}{x_i-x_{i_s}}
\,+\,
\frac{x_{i_s}\,\Psi(x_i,y_{j_s})}{x_{i_s}-x_i}
\,\biggr)
\ \,\prod_{r\neq s}\ \Psi(x_{i_r},y_{j_r})\,.
\end{gather}

The summands in the first line of the display
\eqref{33} do not depend on $s\,$.
Hence their sum equals \eqref{20}.  
Let us show that the sum appearing in the second and the third
lines of \eqref{33} equals zero. 
By opening the brackets in the third line
and then swapping the running  indices $i,i_s$ in each term 
coming from the second fraction in the brackets,
the sum in the second and the third lines becomes
\begin{gather*}
\sum_{i=1}^N\ 
\sum_{\substack{i_1\lcd i_k\neq i\\j_1\lcd j_k}}
\,\sum_{s=1}^k\ 
\Bigl(\ \,\prod_{r=1}^k\ \Psi(x_{i_r},y_{j_r})\,\Bigr)\,\times
\\[4pt]
\frac{
(-1)^{i+1}\,x_i^{\,N-1}\,x_{i_s}\,\De_{\ts i}+
(-1)^{i_s+1}\,x_{i_s}^{\,N-1}\,x_i\,\De_{\ts i_s}}
{x_i-x_{i_s}}
\end{gather*}
\\
Here in the first line the product over $r=1\lcd k$ depends 
neither on the index $s$ nor
on the choice of the index $i\neq i_1\lcd i_k\,$. 
The fraction in the second line
does not depend on the indices $j_1\lcd j_k\,$. 
We will show that for any {\it fixed\/} 
distinct indices $i_1\lcd i_k\in\{1\lcd N\}$
the sum of these fractions over $s=1\lcd k$ and 
$i\neq i_1\lcd i_k$ is equal to zero.
This will complete our proof of the identity \eqref{8}.

Let $\si$ range over  
$\Sg_N\,$. 
By the definition of the Vandermonde polynomial, the sum
of the last displayed fractions over $s=1\lcd k$ 
and $i\neq i_1\lcd i_k$ equals
$$
\sum_{i\neq i_1\lcd i_k}\,\sum_{s=1}^k\ \ \sum_{\si(i)=1}
\,(-1)^\si\ 
\frac{x_i^{\,N-1}\,x_{i_s}^{\,N-\si(i_s)+1}-x_{i_s}^{\,N-1}\,x_i^{\,N-\si(i_s)+1}}
{x_i-x_{i_s}}\,\prod_{j\,\neq\,i,i_s} x_j^{\,N-\si(j)}
$$

\medskip\noindent
which is in turn equal to the sum
\begin{equation}
\label{66}
\sum_{i\neq i_1\lcd i_k}\,\sum_{s=1}^k\ \ \sum_{\si(i)=1}
\,(-1)^\si
\sum_{r=2}^{\si(i_s)-1}
x_i^{\,N-r}\,x_{i_s}^{\,N-\si(i_s)+r-1}
\,\prod_{j\,\neq\,i,i_s} x_j^{\,N-\si(j)}\,.
\end{equation}

\medskip
Recall that here $N>1\,$. One can easily prove by induction on 
$N=2,3,\ldots$ that the total number of terms in the sum \eqref{66} equals
$$
(N-1)!\,(N-2)\,(N-k)\,k\,/\,2\,.
$$
However, we shall not use this equality and will leave its proof to the reader. 
In the next subsection, we will prove that all terms in \eqref{66} 
cancel each other.

%=============================================================================

\subsection{Cancellations\/}
\label{cancel}

The sum \eqref{66} is taken over quadruples $(i,s,\si,r)\,$.
Take any of them such~that
$$
\si^{-1}(r)\,\neq\,i_1\lcd i_k\,.
$$
Such a quadruple will be called {\it of type I\/}.
Denote $\si^{-1}(r)=\ib\,$. 
Observe that $\ib\neq i$ because $\si(\ib)=r\ge2$ while $\si(i)=1\,$.
Put
\begin{equation}
\label{sib}
\sib=\si\,\tau_{\,i\ib}
%=\tau_{\,1r}\,\si
\end{equation}
where $\tau_{\,i\ib\,}\in\Sg_N$ is the transposition of $i$ and $\ib\,$.
Then $\sib(i)=r$ and $\sib(\ib)=1\,$.
We also have $\sib(i_s)=\si(i_s)$ because $i_s\neq i,\ib\,$.
Hence the quadruple $(\ib,s,\sib,r)$ appears in \eqref{66}
together with the quadruple $(i,s,\si,r)\,$.
But the summands in~\eqref{66} 
corresponding to these two quadruples cancel each other. Indeed, 

\begin{align*}
-\,(-1)^{\si}\,
x_i^{\,N-r}\,
x_{\ib}^{\,N-\si(\ib)}\,
x_{i_s}^{\,N-\si(i_s)+r-1}\,
&\prod_{j\,\neq\,i,\ib,i_s} x_j^{\,N-\si(j)}\ =
\\
(-1)^{\sib}\,
x_i^{\,N-\sib(i)}\,
x_{\ib}^{\,N-r}\,
x_{i_s}^{\,N-\sib(i_s)+r-1}\,
&\prod_{j\,\neq\,i,\ib,i_s} x_j^{\,N-\sib(j)}\,.
\end{align*}
Note that here
$$
\sib^{-1}(r)=i\neq i_1\lcd i_k
$$ 
hence $(\ib,s,\sib,r)$ is also of type~I\,.
Moreover, by applying our construction to the latter quadruple 
instead of $(i,s,\si,r)$
we get~the initial quadruple $(i,s,\si,r)$ back.

Next take a quadruple $(i,s,\si,r)$ showing in \eqref{66} 
such~that for some index~$\sb$
$$
\si(i_s)-r+1=\si(i_{\sb})\,.
$$
Such a quadruple will be called {\it of type II\/}. Note that 
$s\neq\sb$ because $r\neq1\,$. Put
$$
\sib=\si\,\tau_{\,i_s i_{\sb}}
\,.
$$
Here $\sib(i)=\si(i)=1$ because $i\neq i_s\,,i_{\sb}\,$.
We also have $\sib(i_{\sb})=\si(i_s)\,.$
Hence the quadruple $(i,\sb,\sib,r)$ appears in \eqref{66}
together with $(i,s,\si,r)\,$.
The summands in~\eqref{66} corresponding to 
these two quadruples cancel each other. Indeed, 

\begin{align*}
-\,(-1)^{\si}\,
x_i^{\,N-r}\,
x_{i_s}^{\,N-\si(i_s)+r-1}\,
x_{i_{\sb}}^{\,N-\si(i_{\sb})}\,
&\prod_{j\,\neq\,i,i_s,i_{\sb}} x_j^{\,N-\si(j)}\ =
\\
(-1)^{\sib}\,
x_i^{\,N-r}\,
x_{i_s}^{\,N-\sib(i_s)}\,
x_{i_{\sb}}^{\,N-\sib(i_{\sb})+r-1}\,
&\prod_{j\,\neq\,i,i_s,i_{\sb}} x_j^{\,N-\sib(j)}\,.
\end{align*}
Here 
$$
\sib(i_{\sb})-r+1=\si(i_s)-r+1=\sib(i_s)
$$
and $(i,\sb,\sib,r)$ is also of type~II\,.
Moreover, by applying our construction to the latter quadruple  
instead of $(i,s,\si,r)$ we get~the initial quadruple $(i,s,\si,r)$ back.

Note that the above two constructions differ 
for the quadruples of type I~and~II\,,
while a quadruple can be of both types simultaneously.
However, the summands in \eqref{66} corresponding
to quadruples of any of the two types still cancel each other.
Indeed, take any quadruple $(i,s,\si,r)$ of type I which also has
type II\,. By applying our first constuction to it we get 
another quadruple $(\ib,s,\sib,r)$ of type~I\,, where $\sib$ is defined by
\eqref{sib} while $\ib=\si^{-1}(r)\,$. But then for a certain index $\sb$ 
we have
$$
\sib(i_s)-r+1=\si(i_s)-r+1=\si(i_{\sb})
$$
because $(i,s,\si,r)$ also has type II\,.
Hence the quadruple $(\ib,s,\sib,r)$ is of type II too. 
We could similarly check that the result of applying
our second construction to $(i,s,\si,r)$ is not only of type II\, 
but of type I as well.
However, this is already not needed for the cancellation. 
So we will leave checking it to the reader.

Finally, take any quadruple $(i,s,\si,r)$ appearing in \eqref{66} which is 
neither of type I nor of type II\,.
Such a quadruple will be called {\it of type III\/}. Here $r=\si(i_{\sb})$ 
for some index $\sb$ 
because $(i,s,\si,r)$ is not of type I\,. For some index 
$\ib\neq i_1\lcd i_k$ we~also~have
\begin{equation}
\label{ib}
\si(i_s)-r+1=\si(\ib)
\end{equation}
because  $(i,s,\si,r)$ is not of type II\,. Observe that here
the four indices $\,i,i_{\sb},i_s,\ib\,$ are pairwise distinct. Indeed, 
here we have 
$i,\ib\neq i_s,i_{\sb}$ by definition. Further, here $i\neq\ib$ because 
$\si(i)=1$ while
$\si(\ib)\ge2$ by the definition \eqref{ib}. Furthermore, 
here $i_s\neq i_{\sb}$
because $r<\si(i_s)$ while $r=\si(i_{\sb})\,$.
Put $\sib=\si\,\tau$ where 
$\tau\in\Sg_N$ cyclically permutes the indices $\,i,i_{\sb},i_s,\ib\,$
and leaves all the remaining indices fixed. More~exactly,
$$
\tau:\,i\mapsto i_{\sb}\mapsto i_s\mapsto\ib\mapsto i\,.
$$

Let $\rb$ be the number at either side of equality \eqref{ib}.
Consider the quadruple $(\ib,\sb,\sib,\rb)\,$. Here $\sib(\ib)=\si(i)=1$ 
by definition.
Due to the range of $r$ we also have 
$$
2\le\rb\le\si(i_s)-1=\sib(i_{\sb})-1\,.
$$
Therefore the quadruple $(\ib,\sb,\sib,\rb)$ appears in \eqref{66}
together with $(i,s,\si,r)\,$.
The summands in~\eqref{66} corresponding to the two quadruples 
cancel each other. Indeed, we have the equality

\begin{align*}
-\,(-1)^{\si}\,
x_i^{\,N-r}\,
x_{i_{\sb}}^{\,N-\si(i_{\sb})}\,
x_{i_s}^{\,N-\si(i_s)+r-1}\,
x_{\ib}^{\,N-\si(\ib)}\,
&\prod_{j\,\neq\,i,i_{\sb},i_s,\ib}
x_j^{\,N-\si(j)}\ =
\\
(-1)^{\sib}\,
x_i^{\,N-\sib(i)}\,
x_{i_{\sb}}^{\,N-\sib(i_{\sb})+\rb-1}\,
x_{i_s}^{\,N-\sib(i_s)}\,
x_{\ib}^{\,N-\rb}\,
&\prod_{j\,\neq\,i,i_{\sb},i_s,\ib}
x_j^{\,N-\sib(j)}\,.
\end{align*}
Here $\rb=\sib(i_s)$
so that $(\ib,\sb,\sib,\rb)$ is not of type I\,.
We also have
$$
\sib(i_{\sb})-\rb+1=\sib(i)
$$
so that $(\ib,\sb,\sib,\rb)$ is not of type II\,. 
So this quadruple is of type III.
Moreover, by applying our third construction to this quadruple  
instead of $(i,s,\si,r)$ we get~the initial quadruple $(i,s,\si,r)$ back.
Thus all summands in \eqref{66} cancel each other.~We
have now completed the induction step in the proof of our proposition.

%=============================================================================

\section*{\normalsize\bf Acknowledgements}

This work originates from our discussions with S.\,M.\,Khoroshkin.
It was he who suggested to us to compute limits 
of commuting operators whose eigenvectors are
the Jack symmetric polynomials.
We are also grateful to A.\,K.\,Pogrebkov
for helpful comments on his work \cite{P},
and to E.\,Vasserot for drawing our attention to his
joint work with O.\,Schiffmann \cite{SV}.
The first and second named of us have been supported
by the EPSRC grants EP\ns/I\ts014071 and EP\ns/H000054~respectively.

%=============================================================================

%=============================================================================


\begin{thebibliography}{111}

\bibitem[1]{AW}
A.\,G.\,Abanov and P.\,B.\,Wiegmann,
\emph{Quantum Hydrodynamics, the quantum Benjamin-Ono equation, 
and the Calogero model},
Phys.\ Rev.\ Lett.\ {\bf 95} (2005), 076402.  

\bibitem[2]{AK}
H.\,Awata and H.\,Kanno, 
\emph{Macdonald operators and homological invariants of the colored Hopf link},
J. Phys. {\bf A\,44} (2011), 375201.

\bibitem[3]{AMOS}
H.\,Awata, Y.\,Matsuo, S.\,Odake and J.\,Shiraishi,
\emph{Collective fields, 
Calogero-Sutherland model and generalized matrix models}, 
Phys. Lett. {\bf B\,347} (1995), 49--55. 

\bibitem[4]{CJ}
W.\,Cai and N.\,Jing,
\emph{Applications of Laplace\,-Beltrami operator for Jack polynomials},
European J.\ Combin.\ {\bf 33} (2012), 556--571.

\bibitem[5]{C}
F.\,Calogero,
\emph{Ground state of a one-dimensional N-body system}, 
J.\ Math.\ Phys.\ {\bf 10} (1969), 2197--2200.

\bibitem[6]{D}
A.\,Debiard, 
\emph{Polyn\^omes de Tch\'ebychev et de Jacobi dans un espace 
euclidien de dimension $p$\,},
C.\ R.\ Acad.\ Sc.\ Paris {\bf I\,296} (1983), 529--532.

\bibitem[7]{LV}
L.\,Lapointe and L.\,Vinet,
\emph{Exact operator solution of the Calogero-Sutherland model},
Comm.\ Math.\ Phys.\ {\bf 178} (1996), 425--152.

\bibitem[8]{M}
I.\,G.\,Macdonald, 
{\it Symmetric Functions and Hall Polynomials},
Oxford University Press, 1995.

\bibitem[9]{MP}
J.\,A.\,Minahan and A.\,P.\,Polychronakos,
\emph{Density correlation functions in Calogero-Suther\-land models},
Phys.\ Rev.\ {\bf B\,50} (1994), 4236--4239.

\bibitem[10]{NS}
M.\,L.\,Nazarov and E.\,K.\,Sklyanin,
\emph{Macdonald operators at infinity}, 
{\tt arXiv:1212.2960}

\bibitem[11]{OP}
A.\,Okounkov and R.\,Pandharipande,
\emph{Quantum cohomology of the Hilbert scheme of points in the plane},
Inv.\ Math.\ {\bf 179} (2010), 523--557.

\bibitem[12]{P}
A.\,K.\,Pogrebkov,
\emph{Boson-fermion correspondence and quantum 
integrable and dispersionless models,}
Russ.\ Math.\ Surv.\  {\bf 58} (2003), 1003--1037. 
% DOI 10.1070/RM2003v058n05ABEH000668

\bibitem[13]{Pol}
A.\,P.\,Polychronakos,
\emph{Waves and solitons in the continuum 
limit of the Calogero-Suther\-land model},
Phys.\ Rev.\ Lett.\ {\bf 74} (1995), 5153--5157.

\bibitem[14]{R}
P.\,Rossi,
\emph{Gromov\,-Witten invariants of target curves via Symplectic Field Theory},
J.\ Geom.\ Phys.\ {\bf 58} (2008), 931--941.
% DOI: 10.1016/j.geomphys.2008.02.012

\bibitem[15]{SV}
O.\,Schiffmann and E.\,Vasserot,
\emph{Cherednik algebras, W-algebras and the equivariant cohomology 
of the moduli space of instantons on $\mathbb{A}^2$},
{\tt arXiv:1202.2756}

\bibitem[16]{S}
J.\,Sekiguchi,
\emph{Zonal spherical functions on some symmetric spaces},
Publ.\ Res.\ Inst.\ Math.\ Sci. {\bf 12} (1977), 455--459.

\bibitem[17]{Shi}
J.\,Shiraishi,
\emph{A family of integral transformations and
basic hypergeometric series},
Comm.\ Math.\ Phys. {\bf 263} (2006), 439--460.

\bibitem[18]{Sk}
E.\,K.\,Sklyanin,
\emph{Separation of variables. New trends},
Progress Theor.\ Phys.\ Suppl.
{\bf 118} (1995), 35--60.

\bibitem[19]{Su1}
B.\,Sutherland,
\emph{Exact results for a quantum many-body problem in one dimension}, 
Phys.\ Rev.\ {\bf A\,4} (1971), 2019--2021.

\bibitem[20]{Su2}
B.\,Sutherland,
\emph{Exact results for a quantum many-body problem in one dimension II}, 
Phys.\ Rev.\ {\bf A\,5} (1972), 1372--1376.

\end{thebibliography}
\end{document}